\documentclass[12pt]{amsart}
\usepackage{amssymb}
\usepackage{enumerate}
\usepackage{amsthm}
\usepackage{tikz-cd}
\usepackage{amsmath}
\usepackage[mathscr]{euscript}
\DeclareMathOperator{\Eff}{{Eff}}
\DeclareMathOperator{\Nef}{{Nef}}
\DeclareMathOperator{\mult}{{mult}}
\DeclareMathOperator{\rank}{{rank}}
\DeclareMathOperator{\NE}{{NE}}
\DeclareMathOperator{\Bs}{{Bs}}
\makeatletter
\@namedef{subjclassname@2010}{%
  \textup{2010} Mathematics Subject Classification}
\makeatother
\newtheorem{thm}{Theorem}[section]
\newtheorem{lem}[thm]{Lemma}

\newtheorem{corl}[thm]{Corollary}
\newtheorem{xrem}{Remark}

\begin{document}
\baselineskip=17pt

\subjclass[2010]{Primary 14J60; Secondary 14H60, 14J10}
\keywords{Nef cone, Seshadri constants}
\author{Rupam Karmakar}  
\author{Snehajit Misra}
\address{Institute of Mathematical Sciences\\ CIT Campus, Taramani, Chennai 600113, India and Homi Bhabha National Institute, Training School Complex, Anushakti Nagar, Mumbai 400094, India}
\email[Rupam Karmakar]{rupamk@imsc.res.in}

\address{Institute of Mathematical Sciences\\ CIT Campus, Taramani, Chennai 600113, India and Homi Bhabha National Institute, Training School Complex, Anushakti Nagar, Mumbai 400094, India}
\email[Snehajit Misra]{snehajitm@imsc.res.in}
\begin{abstract}
Let $X = \mathbb{P}(E_1) \times_C \mathbb{P}(E_2)$ where $C$ is a smooth curve and let $E_1$, $E_2$ be vector bundles over $C$. In this paper, we extend the results in \cite{K-M-R} by computing the nef cone of $X$ without restriction on the rank 
or semistability of $E_1$ and $E_2$. We also study the Seshadri constants of ample line bundles on $X$.  We calculate the Seshadri constants in some cases and give bounds in some of the remaining cases.
\end{abstract}

\title{Nef cone and Seshadri constants on products of projective bundles over curves}
\maketitle
\vskip 4mm
\section{introduction}
The nef cone $\Nef(X) \subseteq N^1(X)$ of nef divisors on a projective variety $X$ is an important invariant which gives useful information about the projective embeddings of $X$.
The nef cones of various smooth irreducible projective varieties have been studied by many authors in the last few decades ( See \cite{L1} (Section 1.5), \cite{M}, \cite{Fu}, \cite{B-P}, \cite{M-O-H}, \cite{K-M-R} for more details ).
In his paper \cite{M}, Miyaoka found that in characteristic 0, the nef cone of $\mathbb{P}_C(E)$ is determined by the smallest slope of any nonzero torsion free quotient of $E$. Then, \cite{Fu} generalized this to arbitrary co-dimension cycles showing that the 
effective cones of cycles ( and their duals ) on $\mathbb{P}_C(E)$ are determined by the numerical data in the  Harder-Narasimhan filtration of $E$. \cite{B-P} studied the nef cone of divisors on Grassmann bundles $Gr_s(E)$ and flag bundles over smooth curves and 
extended Miyaoka's result to characteristic $p$. In \cite{K-M-R}, nef cones $\Nef(\mathbb{P}_C(E_1) \times_C \mathbb{P}_C(E_2))$ of divisors are computed under the assumption that $E_1$ and $E_2$ are semistable bundles over a smooth curve $C$ and in few other cases
e.g., rank$(E_1) = $ rank$(E_2) = 2$. Note that in \cite{K-M-R}, the cones are 3-dimensional while the literature abounds with 2-dimensional examples (e.g., $\Nef(\mathbb{P}_C(E))$ or $\Nef(Gr_s(E))$ etc.).
In this paper, we extend the results in \cite{K-M-R}, by computing $\Nef(\mathbb{P}_C(E_1) \times_C \mathbb{P}_C(E_2))$ without restriction on the rank or semistability of $E_1$ and $E_2$.

Let $X$ be a smooth complex projective variety and let $L$ be a nef line bundle on $X$. The Seshadri constant of $L$ at $x \in X$ is defined as 
\begin{align*}
 \varepsilon(X,L,x) := \inf_{x \in C} \hspace{1mm}   \bigl\{\frac{L\cdot C}{\mult_xC}\bigr\}
\end{align*}
where the infimum is taken over all irreducible curves in $X$ passing through $x$ having the multiplicity $\mult_xC$ at $x$. One can easily check that it is enough to take the infimum over irreducible and reduced curves $C$. The Seshadri criterion for ampleness 
says that $L$ is ample iff $\varepsilon(X,L,x) >  0$ for all $x \in X$. The Seshadri constant $\varepsilon(X,L,x)$ of a nef line bundle $L$ at a point $x\in X$ is an interesting invariant of $L$ that measures local positivity around $x$ in several ways: some 
numerical, some cohomological via asymptotic jet separation, and even via differential geometry or from arithmetic height theory.

If $L$ is an ample line bundle, then $\varepsilon(X,L,x) \leq \sqrt[n]{L^n}$ for all $x \in X$, where $n$ is the dimension of $X$ and $L^n$ is the  $n$ fold self-intersection of $L$. Hence, $\varepsilon(X,L,x) \in (0,\sqrt[n]{L^n}]$. Usually, Seshadri constants are
very hard to calculate and most of the time,  one tries to give bounds which sharpen the above mentioned bounds. To get an overview of the current research on Seshadri constants, see \cite{B-D-H-K-K-S-S}.

For an ample line bundle $L$ on $X$, we define
\begin{align*}
 \varepsilon(X,L,1) :=  \sup_{x\in X} \hspace{1mm} \bigl\{\varepsilon(X,L,x)\bigr\}
\end{align*}
\begin{align*}
 \varepsilon(X,L) : = \inf_{x\in X} \hspace{1mm} \bigl\{ \varepsilon(X,L,x)\bigr\}
\end{align*}
so that $0 < \varepsilon(X,L) \leq \varepsilon(X,L,x) \leq \varepsilon(X,L,1) \leq \sqrt[n]{L^n}$ for every point $x \in X$.

Miranda (See \cite{L1}, Example 5.2.1) constructs examples on surfaces where Seshadri constants are arbitrarily small. More precisely, he showed that 
given a positive real number $\delta > 0$, there is an algebraic surface $X$ ( which is obtained by blowing up the projective plane $\mathbb{P}^2$ at suitably chosen points ) and an ample line bundle $L$ on $X$ such that $\varepsilon(X,L,x) < \delta$ 
at a particular point $x \in X$. 

Seshadri constants on ruled surfaces $\mathbb{P}_C(E)$ (rank$(E)$ = 2) over a smooth curve $C$ have been studied by many authors (see \cite{Ga}, \cite{Sy}, \cite{H-M} etc. ). More generally, \cite{B-H-N-N} computes the Seshadri constants of ample line bundles on the Grassmann bundle $Gr_r(E)$ over 
a smooth curve $C$ under the assumption that $E$ is an unstable bundle on $C$. In particular, under some suitable conditions on the Harder-Narasimhan filtration of $E$, \cite{B-H-N-N} computes the Seshadri constants of ample line bundles on 
$\mathbb{P}_C(E)$, whenever $E$ is an unstable vector bundle over a smooth curve $C$. Some bounds for the Seshadri constants of ample bundle on ruled surfaces $\mathbb{P}_C(E)$ over smooth curve $C$ are known due to \cite{Ga}. However, the Seshadri constants
of ample line bundles on $\mathbb{P}_C(E)$ are not completely known in the general set up.

Let $E$ be a semistable bundle of rank $r$ and degree $0$ over a smooth curve $C$. Then, $\Nef(\mathbb{P}_C(E)) = \bigl\{ a\xi + bf \mid a, b \in \mathbb{R}_{\geq 0} \bigr \}$, where $\xi$ and $f$ denote the numerical classes of 
$\mathcal{O}_{\mathbb{P}_C(E)}(1)$ and a fibre of the projectivization map respectively. For any point $x \in \mathbb{P}_C(E)$ and any ample line bundle numerically equivalent to $a\xi + bf$, it is known that
$\varepsilon( \mathbb{P}_C(E), a\xi + bf, x) \geq \min\{a, b\}$. 
If $a = \min\{a, b\}$, then the Seshadri constants can be computed by a line through $x$ in the fibre $f$ and $\varepsilon( a\xi + bf, x) = a$. But the case, 
$b = \min\{a, b\}$ is not known.

Most of the research on Seshadri constants is focused on smooth surfaces $X$ ( e.g.; K3 surfaces, surfaces of general type etc.) as well as on smooth projective varieties having Picard rank 2 ( e.g.; $\mathbb{P}_C(E)$, $Gr_r(E)$ ). Also, not much is known for higher dimensional varieties with Picard rank more than 2. Motivated by this,
in this paper, we have calculated the Seshadri constants of ample line bundles on 
$ X = \mathbb{P}(E_1) \times_C \mathbb{P}(E_2)$, where $E_1$ and $E_2$ are vector bundles over a smooth irreducible curve $C$ of rank $r_1$ and $r_2$ respectively, under some assumptions on $E_1$ and $E_2$, and have given bounds in some other cases.

\section{preliminaries}
All the algebraic varieties are assumed to be irreducible and defined over the field of complex numbers, $\mathbb{C}$. 
\subsection{Definitions}
A line bundle $L$ ( a Cartier divisor $D$ ) on an irreducible smooth projective variety $X$ is said to be nef, if $L\cdot C \geq 0$ ( $D\cdot C\geq 0$ respectively ) for all irreducible curves $C\subseteq X$. The \textit{nef cone} $\Nef(X)$ is the convex cone of all nef $\mathbb{R}$-divisor classes on $X$.

Let $X$ be a smooth projective variety and $N_k(X)$ be the real vector space of $k$-cycles on $X$ modulo numerical equivalence. For each $k$, $N_k(X)$ is a finite dimensional real vector space. 
Since $X$ is smooth, we can identify $N_k(X)$ with the abstract dual $N^{n - k}(X) := N_{n - k}(X)^\vee$ via the perfect intersection pairing $N_k(X) \times N_{n - k}(X) \longrightarrow \mathbb{R}$ . 
In particular, $N_1(X)$ is the space of curves and $N^1(X)$ is the real N\'{e}ron-Severi group.

 For any $k$-dimensional subvariety $V$ of $X$, let $[V]$ be its class in $N_k(X)$. A class $\alpha \in N_k(X)$ is said to be effective if there exist subvarieties $V_1, V_2,\cdot\cdot\cdot\cdot\cdot, V_m$ and non-negative real numbers $n_1, n
 _2, ..., n_m$ such that $\alpha$ can be written as $ \alpha = \sum\limits_{i=1}^m n_i[ V_i ]$. The \it pseudo-effective cone \rm  $\overline{\Eff}_k(X) \subset N_k(X)$ is the closure of the cone generated by the classes of effective $k$-cycles in $X$. 

 Note that, $\overline{\Eff}_1(X)$ is the \textit{closed cone of curves}, which is also denoted by $\overline{\NE}(X)$ in the literature.
 
 Let, $D$ be a $\mathbb{Q}$-divisor on $X$. The \textit{stable base locus} of $D$ is
 \begin{align*}
 \mathbf{B}(D) := \bigcap_{m \in \mathbb{N}} \Bs(\mid mD \mid)_{red},
 \end{align*}

where the intersection is taken over all $m$ such that $mD$ is an integral divisor and the \textit{base locus} $\Bs(|D|)$ of a complete linear system $|D|$ of Cartier divisors on $X$ is the set of common zeros of all sections of the associated line bundle $L(D)$.

The \textit{restricted base locus} of a $\mathbb{R}$-divisor $D$ on $X$ is defined to be
\begin{align*}
\mathbf{B}_{-}(D) := \bigcup_{A} \mathbf{B}(D + A),
\end{align*}
where the union is taken over all ample divisors $A$ such that $D + A$ is a $\mathbb{Q}$-divisor.

A vector bundle $E$ of rank $2$ over $C$ is said to be $\textit{normalised}$ ( in the sense of \cite{H}) if $H^0(E) \neq 0$, but for all line bundles $L$ on $C$ with $\deg(L) < 0$, $H^0(E \otimes L) = 0$.

\subsection{Geometry of fibre product of projective bundles over a smooth curve}
Let $E_1$ and $E_2$ be two vector bundles over a smooth curve $C$ of rank $r_1$, $r_2$ and degrees $d_1$, $d_2$ respectively. Let $\mathbb{P}(E_1) = \bf Proj $ $(\oplus_{d \geq 0}Sym^d(E_1))$ and $\mathbb{P}(E_2) = \bf Proj $ $(\oplus_{d \geq 0}Sym^d(E_2))$ be the 
associated projective  bundles together with the projection morphisms $\pi_1 : \mathbb{P}(E_1) \longrightarrow C$ and  $\pi_2 : \mathbb{P}(E_2) \longrightarrow C$ respectively. Let $X = \mathbb{P}(E_1) \times_C \mathbb{P}(E_2)$ be the fibre product over
$C$. Consider the following commutative diagram:
\begin{center}
 \begin{tikzcd} 
 X = \mathbb{P}(E_1) \times_C \mathbb{P}(E_2) \arrow[r, "p_1"] \arrow[d, "p_2"]
& \mathbb{P}(E_2)\arrow[d,"\pi_2"]\\
\mathbb{P}(E_1) \arrow[r, "\pi_1" ]
& C
\end{tikzcd}
\end{center}
Let $f_1,f_2$ and $F$ denote the numerical equivalence classes of the fibres of the maps $\pi_1,\pi_2$ and $\pi_1 \circ p_2 = \pi_2 \circ p_1$ respectively. Note that $X \cong \mathbb{P}(\pi_1^*(E_2)) \cong \mathbb{P}(\pi_2^*(E_1))$.
We first fix the following notations for the numerical equivalence classes,
\vspace{2mm}
\begin{center}
$\eta_1 = \bigl[\mathcal{O}_{\mathbb{P}(E_1)}(1)\bigr] \in N^1(\mathbb{P}(E_1))$  \hspace{3mm}, \hspace{3mm} $\eta_2 = \bigl[\mathcal{O}_{\mathbb{P}(E_2)}(1)\bigr] \in N^1(\mathbb{P}(E_2)),$
\end{center}
 \vspace{2mm}
 \bigskip
 We here summarise some results that have been discussed in \cite{K-M-R} ( See Section 3 in \cite{K-M-R} for more details) :
 
 \begin{center}
  $F = p_2^\ast(f_1) = p_1^ \ast (f_2) \hspace{2.5mm} \hspace{2.5mm} , \hspace{2.5mm} F^2 = 0 \hspace{2.5mm},  \hspace{2.5mm} N^1(X) = \mathbb{R}(p_1^*\eta_2) \oplus \mathbb{R}(p_2^*\eta_1) \oplus \mathbb{R}F \hspace{2.5mm}, $
  \vskip 2mm
  
$(p_1^*\eta_2)^{r_2}\cdot F = 0\hspace{2.5mm},\hspace{2.5mm}  (p_1^*\eta_2)^{r_2 + 1} = 0 \hspace{2.5mm}, \hspace{2.5mm} (p_2^*\eta_1)^{r_1}\cdot F = 0\hspace{2.5mm}, \hspace{2.5mm} (p_2^*\eta_1)^{r_1 + 1} = 0 \hspace{2.5mm}, $ 
  \vskip 2mm

$(p_2^*\eta_1)^{r_1} = (\deg(E_1))F\cdot(p_2^*\eta_1)^{r_1-1}\hspace{2.5mm},  \hspace{2.5mm}   (p_1^*\eta_2)^{r_2} = (\deg(E_2))F\cdot(p_1^*\eta_2)^{r_2-1}\hspace{3.5mm}, $
 \vskip 2mm

$(p_2^*\eta_1)^{r_1}\cdot(p_1^*\eta_2)^{r_2-1} = \deg(E_1)\hspace{3.5mm}, \hspace{3.5mm} (p_1^*\eta_2)^{r_2}\cdot(p_2^*\eta_1)^{r_1-1} = \deg(E_2)\hspace {2mm}.$

 \end{center}

\smallskip
Also, the dual basis of $N_1(X)$ is $\{ \delta_1,\delta_2, \delta_3 \}$ where,
 \begin{center}
  
$\delta_1 = F\cdot(p_2^*\eta_1)^{r_1-2}\cdot(p_1^*\eta_2)^{r_2-1} \hspace{2.5mm},\hspace{2.5mm} \delta_2 = F\cdot(p_2^*\eta_1)^{r_1-1}\cdot(p_1^*\eta_2)^{r_2-2}, $

$\delta_3 = (p_2^*\eta_1)^{r_1-1}\cdot(p_1^*\eta_2)^{r_2-1} - \deg(E_1)F\cdot(p_2^*\eta_1)^{r_1-2}\cdot(p_1^*\eta_2)^{r_2-1} -\deg(E_2)F\cdot(p_2^*\eta_1)^{r_1-1}\cdot(p_1^*\eta_2)^{r_2-2}.$ 

\end{center}
  \section{Nef cones of fibre product}
  Let $C$ be a smooth curve over the field of complex numbers $\mathbb{C}$ and let $E$ be a vector bundle over $C$. The slope of $E$ is defined as 
\begin{center}
$\mu(E) := \frac{\deg E}{r} \in \mathbb{Q}$
\end{center}
A vector bundle $E$ over $C$ is said to be semistable is $\mu(F) \leq \mu(E)$ for all subbundle $F \subseteq E$. For every vector bundle $E$, there is a unique filtration
 \begin{align*}
  E = E_0 \supset E_{1} \supset \cdot\cdot\cdot \supset E_{l-1} \supset E_l = 0
 \end{align*} 
called the Harder-Narasimhan filtration, such that $E_i/E_{i+1}$ is semistable for each $i \in \{0,1,\cdot\cdot\cdot\cdot,l-1\}$ and $\mu(E_i/E_{i+1}) > \mu(E_{i-1}/E_{i})$ for all $i \in \{1,2,\cdot\cdot\cdot,l-1\}$.
See \cite{H-L} for more details on semistability.

Let $E_1$ and $E_2$ be two vector bundles of rank $r_1$ and $r_2$  and degree $d_1$ and $d_2$ respectively over a smooth curve $C$.

\smallskip
Let $E_1$ admits the unique Harder-Narasimhan filtration 
\begin{align*}
E_1 = E_{10} \supset E_{11} \supset ... \supset E_{1l_1} = 0
\end{align*}
with $Q_{1i} := E_{1(i-1)}/ E_{1i}$ being semistable for all $i \in [1,l_1-1]$. Denote $ n_{1i} = \rank(Q_{1i}), \\
d_{1i} = \deg(Q_{1i})$ and $\mu_{1i} = \mu(Q_{1i}) := \frac{d_{1i}}{n_{1i}}$ for all $i$.

Similarly, let $E_2$ also admits the unique Harder-Narasimhan filtration 
\begin{align*}
E_2 = E_{20} \supset E_{21} \supset ... \supset E_{2l_2} = 0
\end{align*}
with $ Q_{2i} := E_{2(i-1)} / E_{2i}$ being semistable for $i \in [1,l_2-1]$. Denote $n_{2i} = \rank(Q_{2i}), \\ 
d_{2i} = \deg(Q_{2i})$ and $\mu_{2i} = \mu(Q_{2i}) := \frac{d_{2i}}{n_{2i}}$ for all $i$.
\bigskip

\begin{thm}
Let $E_1$ and $E_2$ be two vector bundles on a smooth complex projective curve $C$ and let $X = \mathbb{P}(E_1) \times_C \mathbb{P}(E_2)$ as discussed earlier. Then, 
\begin{center}
$\Nef(\mathbb{P}(E_1) \times_C \mathbb{P}(E_2)) = \bigl \{ a\tau_1 + b\tau_2 + cF \mid a, b, c \in  \mathbb{R}_{\geq 0}\bigr\}$.
\end{center}
where $\tau_1 = (p_2^*\eta_1) - \mu_{11}F$ and $\tau_2 = (p_1^*\eta_2) - \mu_{21}F$ and $\mu_{11}$ and $\mu_{21}$ are the smallest slopes of any torsion-free quotients of $E_1$ and $E_2$ respectively, with the same notation as above.
\bigskip
\begin{proof}
By the result of \cite{Fu}, $ \Nef(\mathbb{P}(E_i)) = \bigl \{ a_i(\eta_i - \mu_{i1}f_i) + b_if_i \mid a_i, b_i \in \mathbb{R}_{\geq 0} \bigr\}$ for $i = 1, 2$. 
Since pullback of nef line bundles are nef, we get ,  $ \tau_1 = (p_2^\ast\eta_1) - \mu_{11}F$, $ \tau_2 = (p_1^\ast\eta_2) - \mu_{21}F$ and $F$ are nef.

Now, from the Harder-Narasimhan filtration of $E_i$'s ($i = 1, 2$) as described above, we get the following short exact sequences
\begin{align*}
0 \longrightarrow E_{i1} \longrightarrow E_i \longrightarrow Q_{i1} \longrightarrow 0
\end{align*}
for $i = 1, 2$.

Let  $j_i : \mathbb{P}(Q_{i1}) \longrightarrow \mathbb{P}(E_i)$ denote the canonical embeddings for $i = 1, 2$.

We now proceed along the lines of [Section 2, \cite{Fu}]. The result in [Example 3.2.17, \cite{Ful}] adjusted to bundles of quotients over curves shows that
\begin{align*}
\bigl[\mathbb{P}(Q_{11})\bigr] = \eta_1^{r_1 - n_{11}} + (d_{11} - d_1)\eta_1^{r_1 - n_{11} - 1}f_1 \in N_{n_{11}}(\mathbb{P}(E_1))
\end{align*}
and 
\begin{align*}
\bigl[\mathbb{P}(Q_{21})\bigr] = \eta_2^{r_2 - n_{21}} + (d_{21} - d_2)\eta_2^{r_2 - n_{21} - 1}f_2 \in N_{n_{21}}(\mathbb{P}(E_2))
\end{align*}
where $n_{11} = \rank(Q_{11})$, $n_{21} = \rank(Q_{21}), d_{11} = \deg(Q_{11})$ and $d_{21} = \deg(Q_{21})$.

\smallskip
As ($\eta_1 - \mu_{11}f_1$) and ($\eta_2 - \mu_{21}f_2$) are both nef divisors, we have
\vskip1mm

$\theta_{11} : = \bigl[\mathbb{P}(Q_{11})\bigr] \cdot \bigl(\eta_1 - \mu_{11}f_1\bigr)^{n_{11} - 1} $

\vskip1mm

$ = \bigl\{\eta_1^{r_1 - n_{11}} + (d_{11} - d_1)\eta_1^{r_1 - n_{11} - 1}f_1\bigr\}\cdot\bigl(\eta_1 - \mu_{11}f_1\bigr)^{n_{11} - 1} \in \overline{\Eff}_1(\mathbb{P}(E_1))$

\vskip2mm
and 
\vskip 1mm

$\theta_{21} := \bigl[\mathbb{P}(Q_{21})\bigr] \cdot \bigl(\eta_2 - \mu_{21}f_2\bigr)^{n_{21} - 1}$
\vskip1mm
$= \bigl\{\eta_2^{r_2 - n_{21}} + (d_{21} - d_2)\eta_2^{r_2 - n_{21} - 1}f_2\bigr\}\cdot\bigl(\eta_2 - \mu_{21}f_2\bigr)^{n_{21} - 1} \in \overline{\Eff}_1(\mathbb{P}(E_2))$.
\vskip2mm

Note that, $p_1$ and $p_2$ are proper, flat morphisms, and as the base space is smooth, $p_1$, $p_2$ are also smooth. Hence, numerical pullbacks of cycles are well defined and the flatness of $p_1$ and $p_2$
ensure that pullbacks of numerical classes preserve the pseudo-effectivity. 
We  consider 
$ D : = p_2^\ast(\theta_{11}) \cdot p_1^\ast(\theta_{21})$, which is equal to
\begin{align*}
p_2^\ast\bigl[\mathbb{P}(Q_{11})\bigr] \cdot p_1^\ast\bigl[\mathbb{P}(Q_{21})\bigr] \cdot \bigl\{(p_2^*\eta_1) - \mu_{11}F\bigr\}^{n_{11} - 1} \cdot \bigl\{(p_1^*\eta_2) - \mu_{21}F\bigr\}^{n_{21} - 1}
\end{align*}

 By using the above descriptions of $\theta_{11}$ and $\theta_{21}$, $D$ can be written as
\smallskip

$D= \bigl\{ (p_2^\ast\eta_1)^{r_1 - n_{11}} + (d_{11} - d_1)F \cdot (p_2^\ast\eta_1)^{r_1 - n_{11} - 1} \bigr \} \cdot \bigl \{ (p_1^\ast\eta_2)^{r_2 - n_{21}} + (d_{21} - d_2)F \cdot (p_1^\ast\eta_2)^{r_2 - n_{21} - 1} \bigr \} \cdot$ \\ 
$\bigl\{(p_2^*\eta_1) - \mu_{11}F\bigr\}^{n_{11} - 1} \cdot \bigl\{(p_1^*\eta_2) - \mu_{21}F\bigr\}^{n_{21} - 1}$
\smallskip

$= \bigl \{ (p_2^\ast\eta_1)^{r_1 - 1} + (\mu_{11} - d_1)F  \cdot (p_2^\ast\eta_1)^{r_1 - 2} \bigr \} \cdot \bigl \{ (p_1^\ast \eta_2)^{r_2 - 1} + (\mu_{21} - d_2)F \cdot (p_1^\ast \eta_2)^{r_2 - 2} \bigr \}$
\smallskip

 $= (p_2^*\eta_1)^{r_1 -1} \cdot (p_1^*\eta_2)^{r_2 -1} + (\mu_{11} - d_1)F \cdot (p_2^*\eta_1)^{r_1 - 2} \cdot (p_1^*\eta_2)^{r_2 - 1} + (\mu_{21} - d_2)F \cdot (p_1^*\eta_2)^{r_2 - 2} \cdot (p_2^*\eta_1)^{r_1 - 1}$
 \smallskip
 
  which is clearly a 1-cycle in $X$. Now, $p_2^\ast\bigl[\mathbb{P}(Q_{11})\bigr] \cdot p_1^\ast\bigl[\mathbb{P}(Q_{21})\bigr] = \bigl[\mathbb{P}(Q_{11}) \times_C \mathbb{P}(Q_{21})\bigr]$ is an effective cycle in $X$ , and
  $ (p_2^*\eta_1) - \mu_{11}F , (p_1^*\eta_2) - \mu_{21}F$ are nef divisors in $X$.
  Hence, $D \in \overline{\Eff}_1(X)$.
  \vskip 2mm
  
  Since $ \tau_1\cdot D = \bigl\{(p_2^*\eta_1) - \mu_{11}F\bigr\}\cdot D = 0 , \tau_2\cdot D = \bigl\{(p_1^*\eta_2) - \mu_{21}F\bigr\}\cdot D = 0 $ and $F^2=0$ , $\tau_1 , \tau_2 , F$ are in the boundary of $\Nef(X)$.
  
  If $ a\tau_1 + b\tau_2 +cF$ is any element in $\Nef(X)$, then $(a\tau_1 + b\tau_2 +cF) \cdot D \geq 0$ , which implies that $c \geq 0$. Also, $F \cdot \tau_1^{r_1 - 2} \cdot \tau_2^{r_2 - 1}$ and 
  $F \cdot \tau_1^{r_1 - 1} \cdot\tau_2^{r_2 - 2}$ are intersections of nef divisors. Now
  \begin{align*}
  (a\tau_1 + b\tau_2 +cF)\cdot (F \cdot \tau_1^{r_1 - 2} \cdot \tau_2^{r_2 - 1})& = a F \cdot \tau_1^{r_1 - 1} \cdot \tau_2^{r_2 - 1} + b F \cdot \tau_1^{r_1 - 2} \cdot \tau_2^{r_2} + cF^2 \cdot \tau_1^{r_1 - 2}\cdot \tau_2^{r_2 - 1} \\
  & = aF \cdot (p_2^*\eta_1)^{r_1 -1} \cdot  (p_1^*\eta_2)^{r_2 -1} + bF \cdot(p_2^*\eta_1)^{r_1 - 2} \cdot (p_1^*\eta_2)^{r_2} + 0 \\
 & = a + 0 + 0 \\
 &= a
  \end{align*} 
  and
  \begin{align*}
   (a\tau_1 + b\tau_2 +cF)\cdot (F \cdot \tau_1^{r_1 - 1} \cdot \tau_2^{r_2 - 2})& = a F \cdot \tau_1^{r_1} \cdot \tau_2^{r_2 - 2} + b F \cdot \tau_1^{r_1 - 1} \cdot \tau_2^{r_2 - 1} + cF^2 \cdot \tau_1^{r_1 - 1}\cdot \tau_2^{r_2 - 2} \\
 & = b + 0 + 0 \\
 &= b
 \end{align*}
   Since, $ a\tau_1 + b\tau_2 +cF \in \Nef(X)$, we have $a \geq 0 , b \geq 0$. This completes the proof.
 \end{proof}
\end{thm}

\begin{corl}
Assume that the hypotheses of Theorem 3.1 holds. Then, the closed cone of curves of $X$ is given by 
\begin{center}
$\overline{\NE}(X) = \bigl\{ p\delta_1 + q\delta_2 + r( \delta_3 + \mu_{11}\delta_1 + \mu_{21}\delta_2) \mid p, q, r  \in \mathbb{R}_{\geq 0}\bigr\}$.
\end{center}
\end{corl}

\begin{xrem}
 If $E_1$ and $E_2$ both are semistable bundles in Theorem 3.1 , then for each $ i \in\{1,2\}$, $\mathbb{P}(Q_{i1}) \subset \mathbb{P}(E_i)$  becomes an equality and by putting $\mu_1$ and $\mu_2 , (\mu_i = \mu(E_i)$, $i = 1, 2 )$ in place of 
 $\mu_{11}$ and $\mu_{21}$ in the description above, we recover an earlier result in \cite{K-M-R} $( $ see Theorem 4.1 in \cite{K-M-R}$ )$. Similar alterations can be made if one of the vector bundles is semistable and the other is unstable.
\end{xrem}

\section{Seshadri Constants}
In this section, we will compute the Seshadri constants of ample line bundles on $X = \mathbb{P}(E_1) \times_C \mathbb{P}(E_2)$ in certain cases and will give bounds in some other cases. See the introduction for the definition of Seshadri constant.
 
 \begin{thm}
 Let $E_1$ and $E_2$ be two vector bundles on a smooth curve $C$ with $\mu_{11}$ and $\mu_{21}$ being the smallest slopes of any torsion-free quotient of $E_1$ and $E_2$  respectively and let $X = \mathbb{P}(E_1) \times_C \mathbb{P}(E_2)$. 
Let $L$  be an ample line bundle on $X$ which is numerically equivalent to $a\tau_1 + b\tau_2 + cF \in N^1(X)$. Then, the Seshadri constants of $L$ satisfy,
\begin{align*}
 \varepsilon(X, L, x) \geq \min \{a, b, c \} ,  \hskip 2mm \forall x \in X. 
\end{align*}
Moreover, 

 (4.1.1) if $a = \min\{a,b,c\},$ then $\varepsilon(X, L, x) = a, \hskip 2mm \forall x \in X$

 (4.1.2) if $b = \min\{a,b,c\},$ then $ \varepsilon(X, L, x) = b, \hskip 2mm \forall x \in X$.

\end{thm}
\bigskip
Before going into the proof of the Theorem 4.1, we will prove the following useful lemma.
\begin{lem}
Let $L$ be an $\mathbb{R}$-divisor of type $(a, b)$ on $\mathbb{P}^n \times \mathbb{P}^m$, with $a, b \in \mathbb{R}_{\geq 0}$. Then,
\begin{align*}
\varepsilon( \mathbb{P}^n \times \mathbb{P}^m , L, p) = \min\{ a, b \} \hskip 2mm \forall p \in \mathbb{P}^n \times \mathbb{P}^m
\end{align*}
\end{lem}
\begin{proof}
Let $B$ be an irreducible curve in $\mathbb{P}^n \times \mathbb{P}^m$. Then, $B$ can be written as $B = x(1, 0)^{n - 1} + y(0, 1)^{m - 1}$ for some $x, y \in \mathbb{R}_{\geq 0}$.
Also, for any $p \in \mathbb{P}^n \times \mathbb{P}^m$, we have  $ \deg B \geq \mult_pB$. Hence,
\begin{align*}
\frac{L \cdot B}{\mult_p B} = \frac{ay + bx}{\mult_p B} \geq \min \{ a, b \} \cdot \frac{y + x}{\mult_p B} \geq \min\{ a, b \}
\end{align*}
\smallskip
Now, for any point $p \in \mathbb{P}^n \times \mathbb{P}^m$, write $ p = (p_1, p_2)$, with $p_1 \in \mathbb{P}^n$ and $ p_2 \in \mathbb{P}^m$. Then, $p \in p_1 \times l_2$ and $p \in l_1 \times p_2$, 
where $l_1$ and $l_2$ are classes of lines in $\mathbb{P}^n$ and $\mathbb{P}^m$ respectively. This gives us 
\begin{align*}
\varepsilon(\mathbb{P}^n \times \mathbb{P}^m, L , p) \leq \frac{ L \cdot (p_1 \times l_2)}{1} = a \quad and \quad \varepsilon(\mathbb{P}^n \times \mathbb{P}^m, L , p) \leq \frac{L \cdot (l_1 \times p_2)}{1} = b
\end{align*}
 which implies $\varepsilon(\mathbb{P}^n \times \mathbb{P}^m, L , p) \leq \min\{ a, b \}$. This proves the lemma.
\end{proof}

\vskip 2mm

\begin{proof}
\it of Theorem 4.1 \rm :
By Theorem 3.1 and Corollary 3.2, 
\begin{align*}
\Nef(X) = \bigl\{ a\tau_1 + b\tau_2 + cF \mid a,b,c \in \mathbb{R}_{\geq 0} \bigr\}
\end{align*}
and 
\begin{align*}
 \overline{\NE}(X) = \bigl\{ p\delta_1 + q\delta_2 + r\overline{\delta_3} \in N_1(X) \mid p,q,r\in \mathbb{R}_{\geq 0} \bigr\},
\end{align*}
 where $\overline{\delta_3} = \delta_3 + \mu_{11}\delta_1 + \mu_{21}\delta_2$.

 Let $B$ be a reduced and irreducible curve passing through $x \in X$ with multiplicity $m$ at $x \in X$. Then $B$ can be written as $B =  p\delta_1 + q\delta_2 +r\overline{\delta_3} \in \overline{\NE}(X) \subseteq N_1(X)$ . Two cases can occur :
 \subsection*{Case I}
 Assume that $B$ is not contained in any fibre of the map $(\pi_1\circ p_2)$ over the curve $C$. Hence, by B\'{e}zout's Theorem :
\begin{align}\label{seq1}
  F\cdot B \geq \mult_xB = m
\end{align}
This implies, $r\geq m$. Since $L$ is ample, $a,b,c > 0$.
Hence,
\begin{align*}
 \frac{L\cdot B}{\mult_xB} = \frac{L\cdot B}{m} = \frac{(a\tau_1 + b\tau_2 + cF)\cdot(p\delta_1 + q\delta_2 +r\overline{\delta_3})}{m}
\end{align*}

\begin{align*} 
= \frac{ap + bq}{m} + c\cdot\frac{r}{m} \geq c\cdot\frac{r}{m} \geq c.
 \end{align*}

\subsection*{Case II}
Assume that $B$ is contained in some fibre $F$ of the map $(\pi_1\circ p_2)$ over the curve $C$. Hence, $F\cdot B = 0$ which implies $r=0$. 
We know that the fibres of the map $(\pi_1\circ p_2)$ are isomorphic to $\mathbb{P}^{r_1-1} \times \mathbb{P}^{r_2-1}$. Since $B$ is curve in $\mathbb{P}^{r_1-1} \times \mathbb{P}^{r_2-1}$ passing through $x$ of 
multiplicity $m$, then  from Lemma 4.2. , $\frac{L\cdot B}{\mult_xB} \geq \min\{a,b\}$.

Combining both cases, we have, $\varepsilon(X,L,x) := \inf\limits_{x \in C} \hspace{1mm} \{\frac{L\cdot C}{\mult_xC}\} \geq \min\{a,b,c\}$ , $\forall x\in X$.
\bigskip

\hspace*{5mm} Now, a point $x \in X$ can be written as $ x = (x_1, x_2 )$, where $x_1 \in \mathbb{P}(E_1), x_2 \in \mathbb{P}(E_2)$. Take the class of a line $l_2$ in the fibre $f_2$ of $\pi_2$  passing through $x_2$. 
 Then, \\
 $x \in x_1 \times l_2 = \delta_1 \{  = F \cdot (p_2^*\eta_1)^{r_1 - 2} \cdot (p_1^*\eta_2)^{r_2 - 1}\}$ in $N_1(X)$.
 So,
 \begin{align*}
 \varepsilon(X, L, x) \leq \frac{L \cdot \delta_1}{1} = a. 
 \end{align*}
 
 When $ a = \min\{a, b, c\}$, using the above inequality and the fact that $\varepsilon(X,L,x) \geq \min \{a, b, c \}$, we conclude that $\varepsilon(X, L, x) = a$.
 
Similarly, take the class of a line $l_1$ in the fibre $f_1$ of $\pi_1$ passing through $x_1$. Then, \\
$x \in l_1 \times x_2$ = $\delta_2\{ = F \cdot (p_2^*\eta_1)^{r_1 - 1} \cdot (p_1^*\eta_2)^{r_2 - 2}\}$ in $N_1(X)$. So, 
\begin{align*}
\varepsilon(X, L, x) \leq \frac{L \cdot \delta_2}{1} = b.
\end{align*}
 
 So, if $b = \min\{a, b, c\}$, the above inequality and $\varepsilon(X,L,x) \geq \min \{a, b, c\}$ implies that $\varepsilon(X, L, x) = b$. This proves $(4.1.1)$ and $(4.1.2)$.
\end{proof}

In the above theorem if $c= \min\{a, b, c \}$ more can be said about the Seshadri constants of ample line bundles on $\mathbb{P}(E_1) \times \mathbb{P}(E_2)$. These results are explained in the following two theorems.

\begin{thm}
Let $E_1$ and $E_2$ be two unstable vector bundles over a smooth curve $C$ of rank $r_1$ and $r_2$ respectively and $ X = \mathbb{P}(E_1) \times_C \mathbb{P}(E_2)$.
 Let $L$ be an ample line bundle on $X$ numerically equivalent to $ a\tau_1 + b\tau_2 + cF \in N^1(X)$.When $c =\min\{a, b, c\}$ the Seshadri constants of $L$ have the following properties.

 \hspace{5mm} (\romannumeral 1) \, Assume $c \leq a \leq b$, $\rank(E_1) = 2$ and $E_1$ is normalised.
   
  \hspace*{14mm}     If $x$ is a point outside $\mathbf{B}_{-} ( p_2^*\eta_1) $, then $\varepsilon(X,L, x) = a$.
   
 \hspace*{14mm}      If $x$ belongs to $\mathbf{B}_{-}( p_2^*\eta_1)$, then $ c\leq \varepsilon(X, L, x) \leq a$.
 
\hspace*{5mm}  (\romannumeral 2) \,  Assume   $c \leq b \leq a$, $\rank(E_2) = 2$ and $E_2$ is normalised.
   
\hspace*{14mm}    If $x$ is a point outside $\mathbf{B}_{-} (p_1^*\eta_2)$,   then $\varepsilon(X,L, x) = b$.
   
\hspace*{14mm}       If $x$ belongs to $\mathbf{B}_{-}( p_1^*\eta_2)$, the $c \leq \varepsilon(X, L, x) \leq b$.
       
\hspace*{5mm} (\romannumeral 3)      If $x$ is on some curve whose class is proportional to $\overline{\delta_3}$, then $\varepsilon(X, L, x) = c$, where $\overline{\delta_3} = \delta_3 + \mu_{11}\delta_1 + \mu_{21}\delta_2$.
\end{thm}

\begin{proof}
Let $B \subseteq X$ be a reduced and irreducible curve passing through $x \in X$ and $m$ be the multiplicity of $B$ at $x$. Let $B = p\delta_1 + q\delta_2 + r \overline{\delta_3} \in \overline{\NE}(X) \subseteq N_1(X)$,  
where $p, q, r$ are in $\mathbb{R}_{\geq 0}$ and $\overline{\delta_3} = \delta_3 + \mu_{11}\delta_1 + \mu_{21}\delta_2$.
\medskip  

First, assume that $ c \leq a \leq b$. Let $x$ be a point outside of $\mathbf{B}_{-} (p_2^*\eta_1)$.
Then, $B$ is also not contained in $\mathbf{B}_{-}( p_2^*\eta_1)$. Hence, $ p_2^*\eta_1 \cdot B \geq 0$ i.e,
\begin{align*}
p_2^*\eta_1 \cdot ( p\delta_1 + q\delta_2 + r \overline{\delta_3}) \geq 0.
\end{align*}
 which implies,  $p + r\mu_{11} \geq 0$.
 
 Now if $B$ is not contained in the fibre, then by Case(I) in the proof of Theorem 4.1, we get $r\geq m$. Hence,
 \begin{align*}
   \varepsilon(X,L, x) = \frac{ap + bq + cr}{m} \geq \frac{r}{m}(c - a\mu_{11}) + \frac{bq}{m} \geq \frac{r}{m}(c - a\mu_{11}) \geq ( c - a\mu_{11}) \geq - a\mu_{11} \geq a .
\end{align*} 
  ( since $\rank(E_1) = 2$ and $E_1$ is normalised, $\mu(Q_{11}) = \mu_{11} = \deg(Q_{11}) \leq -1$).
  
  And if $B$ is contained in the fibre, then by Case (II) in the proof of Theorem 4.1, we get, $( p+q ) \geq m$. Hence,
  \begin{align*}
  \varepsilon(X,L, x) = \frac{ap + bq}{m} \geq \frac{a(p + q)}{m} \geq a
\end{align*}
as our assumption is $b\geq a \geq c$. We already know that $\varepsilon (X,L,x) \leq a$ from the proof of $(4.1.1)$. So, $\varepsilon(X,L,x) = a$.
If $x$ belongs to $\mathbf{B}_{-}( p_2^*\eta_1)$, then it is obvious that $c \leq \varepsilon(X,L,x) \leq a$.
This completes the proof of $(\romannumeral 1)$. A similar kind of argument will prove $(\romannumeral 2)$.
  
  To prove $(\romannumeral 3)$, observe that $ L\cdot \overline{\delta_3} = c$. So, 
\begin{align*}
  \varepsilon(X, L, x) \leq \frac{L \cdot \overline{\delta_3}}{\mult_x \overline{\delta_3}} \leq \frac{c}{\mult_x \overline{\delta_3}} \leq c.
 \end{align*}
 Therefore, by the above inequality and first part of theorem $4.1$, we get,  $\varepsilon(X, L, x) = c$.
 \end{proof}
 \bigskip
 \begin{corl}
  Assume the hypotheses of  Theorem 4.3 holds and let $L$ be an ample line bundle on $X$ numerically equivalent to $ a\tau_1 + b\tau_2 + cF \in N^1(X)$. Then, we have,
  
  $(\it i) \hspace{2mm} \varepsilon(X,L) = \min\{a,b,c\}$.
  
  $(\it ii) \hspace{2mm} \varepsilon(X,L,1) \leq \min\{a,b\}$.
 \end{corl}
 \begin{proof}
  Since $\varepsilon(X,L,x) \geq \min\{a,b,c\},$ for all $x\in X$, we have, 
  \begin{align*}
  \varepsilon(X,L) = \inf_{x\in X} \hspace{1mm} \varepsilon(X,L,x) \geq \min\{a,b,c\}.
  \end{align*}
  Now, if $ \min\{a,b,c\} = a , \hspace{1mm} \varepsilon(X,L,x) = \varepsilon (X,L)= \min\{a,b,c\} = a , \forall x \in X $. Similarly,
  if $ \min\{a,b,c\} = b,  \varepsilon(X,L,x) = \varepsilon (X,L)= \min\{a,b,c\} = b , \forall x \in X $. Also, when $ \min\{a,b,c\} = c$, then, $\varepsilon(X,L,x) = \min\{a,b,c\} = c \geq \varepsilon (X,L)$, if $x$ is on some curve of class proportional to  $\overline{\delta_3}$. Therefore,
  combining all three cases, we have, $\varepsilon(X,L) = \min\{a,b,c\}$.
  
 In the proof of Theorem 4.1 we have showed that for all $x \in X$, $\varepsilon(X,L,x) \leq \min\{a, b\}$. So, this implies that $\varepsilon(X,L,1)\leq \min\{a, b\}$.
  \end{proof}
\begin{thm}
 Let $E_1$ be a semistable vector bundle of rank $r_1$ and $E_2$ be an unstable vector bundle of rank $r_2$ over a smooth curve $C$ and let $ X = \mathbb{P}(E_1) \times_C \mathbb{P}(E_2)$.
Let $L$ be an ample bundle on $X$ numerically equivalent to $a\tau_1 + b\tau_2 + cF \in N^1(X)$. When $c= \min\{a, b, c\}$  the Seshadri constants of $L$ have the following properties.

 Assume that $c \leq b \leq a$, $\rank(E_2) = 2$ and $E_2$ is normalised.

  (i) if $x$ is a point outside $\mathbf{B}_{-}( p_1^*\eta_2 )$,   then $\varepsilon(X,L, x) = b$.
   
  (ii) if $x$ belongs to $\mathbf{B}_{-}( p_1^*\eta_2)$, then $c \leq \varepsilon(X,L, x) \leq b$.
\end{thm}

\begin{proof}
 The proof is similar to the proof of the Theorem 4.3.
\end{proof}

\subsection*{Acknowledgement}
The authors would like to thank Prof . Krishna Hanumanthu, CMI Chennai and  Prof. D.S. Nagaraj, IISER Tirupati for helpful discussions at every stage of this work. The authors would also like to thank the referee of an earlier version of this article for his valuable comments and useful suggestions towards the overall improvement of the content. This work is supported financially by a fellowship from IMSc,Chennai (HBNI),
DAE, Government of India.

\end{document}